\documentclass[oneside,leqno,11pt]{article}
\usepackage{amssymb,amsmath,latexsym, stmaryrd,amsthm}
\usepackage{amssymb,amsmath,latexsym}

\def\gg{\mathfrak{g}}
\def\ggl{\mathfrak{gl}}
\def\gh{\mathfrak{h}}

\def\gl{\mathfrak{l}}
\def\gm{\mathfrak{m}}

\def\gp{\mathfrak{p}}

\def\gr{\mathfrak{r}}
\def\gs{\mathfrak{s}}
\def\gsl{\mathfrak{sl}}
\def\gso{\mathfrak{so}}
\def\gsu{\mathfrak{su}}
\def\gsp{\mathfrak{sp}}

\def\gu{\mathfrak{u}}

\def\Tr{{\rm Tr}}

\def\C{\mathbb{C}}
\def\D{\mathbb{D}}

\def\F{\mathbb{F}}

\def\H{\mathbb{H}}

\def\R{\mathbb{R}}

\def\Z{\mathbb{Z}}

\def\cC{\mathcal{C}}

\def\cF{\mathcal{F}}
\def\cG{\mathcal{G}}

\newtheorem{theorem}[equation]{Theorem}

\newtheorem{lemma}[equation]{Lemma}

\newtheorem{definition}[equation]{Definition}

\newtheorem{remark}[equation]{Remark}

\title{Parabolic Subgroups of Real Direct Limit Lie Groups}

\begin{document}
\author{Elizabeth Dan-Cohen\thanks{Research partially supported by DFG Grant PE 980/2-1.},
\, Ivan Penkov\thanks{Research partially supported by  DFG Grant  PE 980/2-1.}
\, \& Joseph A. Wolf\thanks{Research partially supported by NSF Grant
DMS 06 52840.
}}
\date{31 December 2008}

\maketitle

\begin{abstract}
Let $G_\R$ be a classical real direct limit Lie group, and $\gg_\R$ its 
Lie algebra.  The parabolic subalgebras of the complexification 
$\gg_\C$ were described by the first two authors.  In the present paper 
we extend these results to $\gg_\R$.  This also gives a description of 
the parabolic subgroups of $G_\R$.  Furthermore, we give 
a geometric criterion for a parabolic subgroup $P_\C$ of $G_\C$ to 
intersect $G_\R$ in a parabolic subgroup.  This criterion involves the
$G_\R$--orbit structure of the flag ind--manifold $G_\C/P_\C$.

\vspace{10pt}
\end{abstract}

\section{Introduction and Basic Definitions} \label{sec1}
\setcounter{equation}{0}

We start with the three classical simple locally finite countable--dimensional Lie 
algebras $\gg_\C = \varinjlim \gg_{n,\C}$, and their real forms
$\gg_\R$.  The Lie algebras $\gg_\C$ are the classical direct limits,
$\gsl(\infty,\C) = \varinjlim \gsl(n;\C)$,
$\gso(\infty,\C) = \varinjlim \gso(2n;\C) = \varinjlim \gso(2n+1;\C)$, and
$\gsp(\infty,\C) = \varinjlim \gsp(n;\C)$,
where the direct systems are
given by the inclusions of the form
$A \mapsto (\begin{smallmatrix} A & 0 \\ 0 & 0 \end{smallmatrix} )$.
See \cite{B} or \cite{BS}.
We often consider the locally reductive algebra 
$\ggl(\infty;\C) = \varinjlim \ggl(n;\C)$ along with $\gsl(\infty;\C)$.

The real forms of these classical simple locally finite countable--dimensional
complex Lie algebras $\gg_\C$ have been classified by A. Baranov in \cite{B}. 
A slight reformulation of \cite[Theorem 1.4]{B} 
says that the following is a complete list of the real forms of $\gg_\C$.

\noindent
\underline{If  $\gg_\C = \gsl(\infty;\C)$},  then $\gg_\R$ is one of the 
following:

$\gsl(\infty;\R) = \varinjlim \gsl(n;\R)$, the real 
special linear Lie algebra, 

$\gsl(\infty;\H) = \varinjlim \gsl(n;\H)$, the quaternionic 
special linear Lie algebra, where
$\gsl(n;\H) := \ggl(n;\H) \cap \gsl(2n;\C)$,  

$\gs\gu(p,\infty) = \varinjlim \gs\gu(p,n)$, the complex special 
unitary Lie algebra of finite real rank $p$,

$\gs\gu(\infty,\infty) = \varinjlim \gs\gu(p,q)$, the complex
special unitary Lie algebra of infinite real rank.
\medskip

\noindent
\underline{If  $\gg_\C = \gso(\infty;\C)$},  then $\gg_\R$ is one of the following:

$\gso(p,\infty) = \varinjlim \gso(p,n)$, the real orthogonal
Lie algebra of finite real rank $p$,

$\gso(\infty,\infty) = \varinjlim \gso(p,q)$, the real
orthogonal Lie algebra of infinite real rank,

$\gso^*(2\infty) = \varinjlim \gso^*(2n)$, 
with
$\gso^*(2n) = \{\xi \in \gsl(n;\H) \mid \kappa_n(\xi x,y) +\kappa_n(x, \xi y)=0 \  \forall x,y \in \H^n \}$, where $ \kappa_n(x,y) := \sum_\ell x^\ell i \bar y^\ell = {}^tx i \bar y$.  Equivalently, $\gso^*(2n) = \gso(2n;\C)\cap 
\gu(n,n)$ with $\gso(2n;\C)$ defined by 
$(u,v) = \sum_1^n (u_{2j-1} v_{2j} + u_{2j}w_{2j-1})$ and $\gu(n,n)$ by
$\langle u,v\rangle = 
\sum_1^n (u_{2j-1} \overline{v_{2j-1}}-u_{2j}\overline{v_{2j}})$. 
\medskip

\noindent
\underline{If  $\gg_\C = \gsp(\infty;\C)$},  then $\gg_\R$ is one of the following:

$\gsp(\infty;\R) = \varinjlim \gsp(n;\R)$, the real
symplectic Lie algebra,

$\gsp(p,\infty) = \varinjlim \gsp(p,n)$, the quaternionic unitary
Lie algebra of finite real rank $p$,

$\gsp(\infty,\infty) = \varinjlim \gsp(p,q)$, the quaternionic
unitary Lie algebra of infinite real rank.
\medskip

\noindent
\underline{If  $\gg_\C = \ggl(\infty;\C)$},  then $\gg_\R$ is one of the following:

$\ggl(\infty;\R) = \varinjlim \ggl(n;\R)$, the real
general linear Lie algebra,

$\ggl(\infty;\H) = \varinjlim \ggl(n;\H)$, the quaternionic
general linear Lie algebra,

$\gu(p,\infty) = \varinjlim \gu(p,n)$, the complex unitary
Lie algebra of finite real rank $p$,

$\gu(\infty,\infty) = \varinjlim \gu(p,q)$, the complex
unitary Lie algebra of infinite real rank.
\medskip

The \emph{defining representations} of $\gg_\C$ are characterized as direct limits of minimal--dimensional nontrivial representations of simple subalgebras.  It is well known that that $\gsl(\infty;\C)$ and $\ggl(\infty;\C)$ have two inequivalent defining representations $V$ and $W$, whereas each of $\gso (\infty;\C)$ and $\gsp(\infty;\C)$ has only one (up to equivalence) $V$.  In particular 
the restrictions to $\gso(\infty;\C)$ or $\gsp(\infty;\C)$ of the two defining 
representations of $\gsl(\infty;\C)$ are equivalent.
The real forms $\gg_\R$ listed above also have \emph{defining representations}, as detailed below, which are particular restrictions of the defining representations of $\gg_\C$.  
We denote an element of $\Z_{\geqq 0} \cup \{ \infty \}$ by $*$.

Suppose that $\gg_\R$ is $\gsl(\infty;\R)$ or $\ggl(\infty;\R)$.  The defining representation spaces of $\gg_\R$ are the finitary (i.e.\ with finitely many nonzero entries) column vectors $V_\R = \R^\infty$ and the finitary row vectors $W_\R = \R^\infty$.   The algebra of $\gg_\R$--endomorphisms of $V_\R$ or $W_\R$ is $\R$.  The restriction of the pairing of $V$ and $W$ is a nondegenerate $\gg_\R$--invariant $\R$--bilinear pairing of $V_\R$ and $W_\R$. 

The defining representation space $V_\R$ of $\gg_\R = \gso(*,\infty)$ 
consists of the finitary real column vectors.  
The algebra of $\gg_\R$--endomorphisms of $V_\R$ 
(the commuting algebra) is $\R$.  The restriction of the symmetric form on $V$ to $V_\R$ is a nondegenerate $\gg_\R$--invariant symmetric $\R$--bilinear form. 

The defining representation space $V_\R$ of $\gg_\R = \gsp(\infty;\R)$ 
consists of the finitary real column vectors.  The algebra of $\gg_\R$--endomorphisms of $V_\R$ is $\R$.  The restriction of the antisymmetric form on $V$ to $V_\R$ is a nondegenerate $\gg_\R$--invariant antisymmetric $\R$--bilinear form.  

In both of these cases the defining representation of $\gg_\R$ is a real form of the defining representation of $\gg_\C$, i.e.\ $V = V_\R \otimes \C$.

Suppose that $\gg_\R$ is $\gsu(*,\infty)$ or $\gu(*,\infty)$.  Then $\gg_\R$ has 
two defining representations, one on the space $V_\R = \C^{*,\infty}$
of finitary complex column vectors and the other on the space
$W_\R$ of finitary complex row vectors.  Thus the two 
defining representations of $\gg_\C$ remain irreducible as a representations
of $\gg_\R$,
the respective algebras of $\gg_\R$--endomorphisms of $V_\R$ and $W_\R$ 
are $\C$, and $V = V_\R$ and $W = W_\R$.  The pairing of $V$ and $W$ defines
a $\gg_\R$--invariant hermitian form of signature $(*,\infty)$ on $V_\R$.  

Suppose that $\gg_\R$ is $\gsl(\infty;\H)$ or $\ggl(\infty;\H)$.  The two defining 
representation spaces of $\gg_\R$ consist of the finitary column vectors $V_\R = \H^\infty$ and finitary row vectors $W_\R = \H^\infty$.  The algebra of $\gg_\R$--endomorphisms of $V_\R$ or $W_\R$ is $\H$.  
The defining representations of $\gg_\C$ on $V$ and $W$ restrict to 
irreducible representations of $\gg_\R$, and $V_\R = \H^\infty = \C^\infty + \C^\infty j  = \C^{2\infty} = V$.  The pairing of $V$ and $W$ is a nondegenerate $\gg_\R$--invariant $\R$--bilinear pairing of $V_\R$ and $W_\R$.  

The defining representation space $V_\R = \H^{*,\infty}$ of $\gsp(*,\infty)$ 
consists of the finitary quaternionic vectors.  The algebra of $\gsp(*,\infty)$--endomorphisms of $V_\R$ is $\H$.  The form on $V_\R$ is a nondegenerate $\gsp(*,\infty)$--invariant quaternionic--hermitian form of signature $(*,\infty)$.  In this case $V_\R = \H^{*,\infty} = \C^{2*,2\infty} = V$.

The defining representation space $V_\R = \H^\infty$ of $\gso^*(2\infty)$ 
consists of the finitary quaternionic vectors.  The algebra of 
$\gso^*(2\infty)$--endomorphisms of $V_\R$ is $\H$.  The form on $V_\R$ is the nondegenerate $\gso^*(2\infty)$--invariant quaternionic--skew--hermitian form $\kappa$ which is the limit of the forms $\kappa_n$.   In this case again $V_\R =  \H^\infty = \C^{2\infty} = V$. 

The Lie ind--group (direct limit group) corresponding to $\ggl(\infty;\C)$ is the general linear group $GL(\infty;\C)$, which consists of all invertible linear transformations of $V$ of the form $g = g' + \mathrm{Id}$ where $g' \in \ggl(\infty;\C)$.  The subgroup of $GL(\infty;\C)$ corresponding to $\gsl(\infty;\C)$ is the special linear group $SL(\infty;\C)$, consisting of elements of determinant $1$.  The connected ind--subgroups of $GL(\infty;\C)$ whose Lie algebras are $\gso(\infty;\C)$ and $\gsp(\infty;\C)$ are denoted by $SO_0(\infty;\C)$ and $Sp(\infty;\C)$.

In Section \ref{sec2} we recall the structure of parabolic subalgebras of complex finitary Lie algebras from \cite{DC2}.  A \emph{parabolic subalgebra} of a complex Lie algebra is by definition a subalgebra that contains a maximal locally solvable (that is, \emph{Borel}) subalgebra.  
Parabolic subalgebras of complex finitary Lie algebras are classified in \cite{DC2}.  
We recall the structural result that every parabolic subalgebra is a subalgebra (technically: defined by infinite trace conditions) of the stabilizer of a 
\emph{taut couple} of generalized flags in the defining representations, and we strengthen this result by studying the non--uniqueness of the flags in the case of the orthogonal Lie algebra.  As in the finite--dimensional case, we define a \emph{parabolic subalgebra} of a real locally reductive Lie algebra $\gg_\R$ as a subalgebra $\gp_\R$ whose complexification $\gp_\C$ is parabolic in $\gg_\C = \gg_\R \otimes_\R \C$.  It is a well--known fact that already in the finite--dimensional case a parabolic subalgebra of $\gg_\R$ does not neccesarily contain a subalgebra whose complexification is a Borel subalgebra of $\gg_\C$.  

In Section \ref{sec3} we prove our main result.  It extends
the classification in \cite{DC2} to the real case.  The key difference from the complex case is that one must take into account the additional structure of a defining representation space of $\gg_\R$ as a module over its algebra of $\gg_\R$--endomorphisms.

In Section \ref{sec5} we give a geometric
criterion for a parabolic subalgebra of $\gg_\C$ to be the complexification
of a parabolic subalgebra of $\gg_\R$.  The criterion is based on an
observation of one of us from the 1960's, concerning the structure of
closed real group orbits on finite--dimensional complex flag manifolds.
We recall that result, appropriately reformulated, and indicate its 
extension to flag ind--manifolds.

\section{Complex Parabolic Subalgebras}\label{sec2}
\setcounter{equation}{0}

\subsection{Generalized Flags}

Let $V$ and $W$ be countable--dimensional right vector spaces over a real division algebra 
$\D$ = $\R$, $\C$ or $\H$, together with
a nondegenerate bilinear pairing
$\langle \cdot , \cdot \rangle : V \times W \rightarrow \D$.  Then $V$ and $W$ are endowed with the Mackey topology, and the closure of a subspace $F \subset V$ is $F^{\perp \perp}$, where $\perp$
refers to the pairing $\langle \cdot , \cdot \rangle$.  
A set of $\D$--subspaces of $V$ (or $W$) is called a \emph{chain} in $V$ (or $W$) if it is 
totally ordered by inclusion.  A \emph{$\D$--generalized flag} is a chain in 
$V$ (or $W$) such that each subspace has an immediate predecessor or an immediate 
successor in the inclusion ordering, and every nonzero vector of $V$ (or $W$) is 
caught between an immediate predecessor--successor pair \cite{DP1}.

\begin{definition} {\rm \cite{DC2}}
A $\D$--generalized flag $\cF$ in $V$ (or $W$) is said to be \emph{semiclosed}
if for every immediate predecessor--successor pair $F' \subset F''$ the closure
of $F'$ is either $F'$ or $F''$. \hfill $\diamondsuit$
\end{definition}

If $\cC$ is a chain in $V$ (or $W$), then we denote by $\cC^\perp$ the chain in $W$ (or $V$)
consisting of the perpendicular complements of the subspaces of $\cC$.

We fix an identification of $V$ and $W$ with the defining representations of $\ggl(\infty;\D)$ as follows.  
To identify $V$ and $W$ with the defining representations of $\ggl(\infty;\D)$, it suffices to find bases in $V$ and $W$ dual with respect to the pairing $\langle \cdot, \cdot \rangle$.  If $\D \neq \H$, the existence of dual bases in $V$ and $W$ with respect to any nondegenerate $\D$--bilinear pairing is a result of Mackey \cite[p. 171]{M}.  Now suppose that $\D = \H$.  Then there exist $\C$--subspaces $V_\C \subset V$ and $W_\C \subset W$ such that $V = V_\C \oplus V_\C j$ and $W = W_\C \oplus W_\C j$.  The restriction of $\langle \cdot , \cdot \rangle$ to $V_\C \times W_\C$ is a nondegenerate $\C$--bilinear pairing.  The result of Mackey 
therefore implies the existence of dual bases in $V_\C$ and $W_\C$, which are also dual bases of $V$ and $W$ over $\H$.  In all cases we identify the right
multiplication of vectors in $V$ by elements of $\D$ with the action of the
algebra of $\gg_\R$--endomorphisms of $V_\R$.

\begin{definition}\label{def-selftaut} {\rm \cite{DC2}}
Let $\cF$ and $\cG$ be $\D$--semiclosed generalized flags in $V$ and $W$,
respectively.  We say $\cF$ and $\cG$ form a \emph{taut couple}
if $\cF^\perp$ is stable under the $\ggl(\infty;\D)$--stabilizer of $\cG$
and $\cG^\perp$ is stable under the $\ggl(\infty;\D)$--stabilizer of $\cF$.
If we have a fixed isomorphism 
 $f : V \rightarrow W$ then we say that $\cF$ is
self--taut if $\cF$ and $f(\cF)$ form a taut couple. 
\hfill $\diamondsuit$
\end{definition}

If one has a fixed isomorphism between $V$ and $W$, then there is an induced bilinear form on $V$.
A semiclosed generalized flag $\cF$ in $V$ is self--taut if and only if $\cF^\perp$ is stable under the $\ggl(\infty;\D)$--stabilizer of $\cF$, where $\cF^\perp$ is taken with respect to the form on $V$.

\begin{remark} \label{selftaut_is_biiso}  {\rm Fix a nondegenerate bilinear form on $V$.  If $V$ is finite dimensional, a self--taut generalized flag in $V$ consists of a finite number of isotropic subspaces together with their perpendicular complements.  In this case, the stabilizer of a self--taut generalized flag equals the stabilizer of its isotropic subspaces.  If $V$ is infinite dimensional, the non--closed non--isotropic subspaces in a self--taut generalized flag in $V$ influence its stabilizer, but it is still true that every subspace is either isotropic or coisotropic. 
Indeed, let $\cF$ be a self--taut generalized flag, and let $F \in \cF$. By \cite[Proposition 3.2]{DC2}, $F^\perp$ is a union
of elements of $\cF$ if it is a nontrivial proper subspace of $V$.  Hence $\cF \cup \{F^\perp\}$ is a chain that
contains both $F$ and $F^\perp$. Thus either $F \subset F^\perp$ or $F^\perp \subset F$, so $F$ is either isotropic or coisotropic.}
\hfill $\diamondsuit$
\end{remark}

We will need the following lemma when we pass to consideration of real parabolic
subalgebras.

\begin{lemma}\label{F-independent}
Suppose that $\D = \H$.  Fix $\H$--generalized flags $\cF$ in $V$ and $\cG$ in $W$.  Then $\cF$ and $\cG$ form a taut couple if and only if they are form taut couple as $\C$--generalized flags.
\end{lemma}

\begin{proof}
It is immediate from the definition that $\cF$ and $\cG$ are semiclosed $\C$--generalized flags if and only if they are semiclosed $\H$--generalized flags.  The proof of \cite[Proposition 3.2]{DC2} holds in the quaternionic
case as well.  Thus if $\cF$ and $\cG$ form a taut couple as either $\C$--generalized flags or $\H$--generalized flags, 
then as long as $F^\perp$ is a nontrivial proper subspace of $W$, it is a union of
elements of $\cG$ for any $F \in \cF$.
Thus $F^\perp$ is stable under both the $\ggl(\infty;\C)$--stabilizer and the $\ggl(\infty;\H)$--stabilizer
of $\cG$ for any $F \in \cF$.  Similarly, if $G \in \cG$ then $G^\perp$ is stable under
both the $\ggl(\infty;\C)$--stabilizer and the $\ggl(\infty;\H)$--stabilizer of $\cF$.  
\end{proof}

\subsection{Trace Conditions}

Let $\gg$ be a locally finite Lie algebra over a field of characteristic
zero.  A subalgebra
of $\gg$ is {\em locally solvable} (resp. {\em locally nilpotent})
if every finite subset of $\gg$ is contained in a solvable (resp. nilpotent)
subalgebra.  The sum of all locally solvable ideals is again a locally
solvable ideal, the {\em locally solvable radical} of $\gg$.  If $\gr$ is the
locally solvable radical of $\gg$ then 
$\gr \cap [\gg,\gg]$ is a locally nilpotent ideal in $\gg$.
Indeed, note that
$\gr \cap [\gg,\gg] = \bigcup_n (\gr \cap [\gg,\gg])\cap \gg_n$
for any exhaustion $\gg = \bigcup_n \gg_n$ by finite--dimensional subalgebras $\gg_n$,
and furthermore $(\gr \cap [\gg,\gg])\cap \gg_n$ is nilpotent for all $n$ by standard
finite--dimensional Lie theory.

Let $\gg$ be a splittable subalgebra of $\ggl(\infty;\D)$, that is, a 
subalgebra containing the Jordan components of its elements), and let 
$\gr$ be 
its locally solvable radical.  The {\em linear nilradical} $\gm$ of $\gg$
is defined to be the set of all nilpotent elements in $\gr$.

\begin{lemma}\label{lin-nilrad}
Let $\gg$ be a splittable subalgebra of $\ggl(\infty;\D)$.  Then its
linear nilradical $\gm$ is a locally nilpotent ideal.  If $\D = \R$, then
 the complexification $\gm_\C$ is the linear nilradical
of $\gg_\C$.
\end{lemma}

\begin{proof}
If $\xi, \eta \in \gm$ they are both contained in the solvable radical
of a finite--dimensional subalgebra of $\gg$, so $\xi + \eta$ and
$[\xi,\eta]$ are nilpotent.  Thus, by Engel's Theorem, $\gm$ is a locally nilpotent subalgebra
of $\gg$.  Although
it is only stated for complex Lie algebras, \cite[Proposition 2.1]{DC2}
shows that $\gm \cap [\gg,\gg] = \gr \cap [\gg,\gg]$, so
$[\gm,\gg]\subset [\gr,\gg] \subset \gr \cap [\gg,\gg]$, and thus $\gm$ is an
ideal in $\gg$.  This proves the first statement.  For the second let $\gr$ be the
locally solvable radical of $\gg$ and note that $\gr_\C$ is the locally
solvable radical of $\gg_\C$, so the assertion follows from finite--dimensional theory. 
\end{proof}

\begin{definition} \label{def-trace-cond}
Let $\gg$ be a splittable subalgebra of $\ggl(\infty;\F)$ where $\F$ is $\R$ or $\C$, and
and let $\gm$ be its linear nilradical.
A subalgebra $\gp$ of $\gg$ is \emph{defined by trace conditions on $\gg$} if 
$\gm \subset \gp$  and  
$$[\gg,\gg]/\gm \subset \gp/\gm \subset \gg/\gm,$$
in other words if there is a family $\Tr$ of Lie algebra homomorphisms
$f: \gg \to \F$ 
with joint kernel equal to $\gp$.  Further, $\gp$ is {\em defined by infinite trace conditions}
if every $f \in \Tr$ annihilates every finite--dimensional simple ideal
in $[\gg,\gg]/\gm$. \hfill $\diamondsuit$
\end{definition}

We write $\Tr^\gp$ for the maximal family $\Tr$ of Definition 
\ref{def-trace-cond}.  On the group level we have corresponding
{\em determinant conditions} and {\em infinite determinant conditions}.
Note that infinite trace conditions and infinite determinant conditions
do not occur when $\gg$ and $G$ are finite dimensional. 

\subsection{Complex Parabolic Subalgebras} \label{sec2C}

Recall that a \emph{parabolic subalgebra} of a complex Lie algebra is by definition a subalgebra that contains a Borel subalgebra, i.e. a maximal locally solvable subalgebra.

\begin{theorem} {\rm \cite{DC2}} \label{complex-parab}
Let $\gg_\C$ be $\ggl(\infty,\C)$ or $\gsl(\infty,\C)$, and let
$V$ and $W$ be its defining representation spaces.
A subalgebra of $\gg_\C$ (resp. subgroup of $G_\C$) is 
parabolic if and only if it is defined by infinite trace conditions 
(resp. infinite determinant conditions) on the 
$\gg_\C$--stabilizer (resp. $G_\C$--stabilizer) of 
a (necessarily unique) taut couple of $\C$--generalized flags $\cF$ 
in $V$ and $\cG$ in $W$.

Let $\gg_\C$ be $\gso(\infty,\C)$ or $\gsp(\infty,\C)$. and let
$V$ be its defining representation space.
A subalgebra of $\gg_\C$ (resp. subgroup of $G_\C$) is 
parabolic if and only if it is defined by infinite trace conditions 
(resp. infinite determinant conditions) on the 
$\gg_\C$--stabilizer (resp. $G_\C$--stabilizer) of 
a self--taut $\C$--generalized flag $\cF$ in $V$.  In the $\gsp(\infty,\C)$ case
the flag $\cF$ is necessarily unique.
\end{theorem}

In contrast to the finite dimensional case, the normalizer
of a parabolic subalgebra can be larger than the 
parabolic algebra.  
 For example, Theorem~\ref{complex-parab} implies that $\gsl(\infty,\C)$ is parabolic in $\ggl(\infty;\C)$, since it is the elements of the stabilizer of the trivial generalized flags  $\{ 0 , V \}$ and $\{ 0 , W \}$ whose usual trace is $0$.
To understand the origins of this example, one should consider the explicit construction in \cite{DP2} of a locally nilpotent Borel subalgebra of $\ggl(\infty;\C)$.
The normalizer of a parabolic subalgebra equals the stabilizer of the corresponding generalized flags \cite{DC2}, which is in general larger than the parabolic subalgebra because of the infinite determinant conditions.   The self--normalizing parabolics are thus those for which $\Tr^\gp = 0$.  
This is in contrast to the finite--dimensional setting, where there are no infinite trace conditions, and all parabolic subalgebras are self--normalizing.

In \cite{DC2} the uniqueness issue was discussed for $\ggl(\infty,\C)$, 
$\gsl(\infty,\C)$, and $\gsp(\infty,\C)$, but not for $\gso(\infty,\C)$.  
In the orthogonal setting one can have three different 
self--taut generalized flags with the same stabilizer 
(see \cite{DC1} and \cite{DP3}, where the non--uniqueness is discussed in 
special cases.)

\begin{theorem}\label{soc-flag-anomaly}
Let $\gp$ be a parabolic subalgebra given by infinite trace conditions on the  $\gso(\infty;\C)$--stabilizer  of a self--taut generalized flag $\cF$ in $V$.  Then there are two possibilities:
\begin{enumerate}
\item $\cF$ is uniquely determined by $\gp$;
\item there are exactly three self-taut generalized flags with the same stabilizer as $\cF$.
\end{enumerate}
The latter case occurs precisely when there exists an isotropic subspace $L \in \cF$ with  $\dim_\C L^\perp / L = 2$.  The three flags with the same stabilizer are then
\begin{itemize}
\item $\{F \in \cF \mid F \subset L \textrm{ or } L^\perp \subset F \}$
\item $\{F \in \cF \mid F \subset L \textrm{ or } L^\perp \subset F \} \cup M_1$
\item $\{F \in \cF \mid F \subset L \textrm{ or } L^\perp \subset F \} \cup M_2$
\end{itemize}
where $M_1$ and $M_2$ are the two maximal isotropic subspaces containing $L$.
\end{theorem}

\begin{proof}
The main part of the proof is to show that $\gp$ determines all the subspaces in $\cF$, except a maximal isotropic subspace under the assumption that $\cF$ has a closed isotropic subspaces $L$ with $\dim_\C L^\perp / L = 2$.  

Let $A$ denote the set of immediate predecessor--successor pairs of $\cF$ such that both subspaces in the pair are isotropic.  Let $F'_\alpha$ denote the predecessor and $F''_\alpha$ the successor 
of each pair $\alpha \in A$.  Let $M$ denote the union of all the isotropic subspaces in $\cF$, i.e.\ $M = \bigcup_{\alpha \in A} F''_\alpha$.  If $M \neq M^\perp$, then $M$ has an immediate successor $W$ in $\cF$.  Note that $W$ is not isotropic, by the definition of $M$.  Furthermore, one has $W^\bot = M$ since $\cF$ is a self--taut generalized flag.  If $M = M^\perp$, let us take $W = 0$.

Let $C$ denote the set of all $\gamma \in A$ such that $F'_\gamma$ is closed.  For each $\gamma \in C$, it is seen in \cite{DC2} that the coisotropic subspace $(F''_\gamma)^\perp$ has an immediate successor in $\cF$.  For each $\gamma \in C$, let $G''_\gamma$ denote the immediate successor of $(F''_\gamma)^\perp$ in $\cF$.  It is also shown in \cite{DC2} that $(G''_\gamma)^\perp = F'_\gamma$.

Since $\cF$ is a self--taut generalized flag, $\cF$ is uniquely determined by the set of subspaces $$\{ F''_\alpha \mid \alpha \in A \} \cup \{ G''_\gamma \mid \gamma \in C \textrm{ such that }G''_\gamma \textrm{ is not closed} \} \cup \{ W \}.$$  We use separate arguments for these three kinds of subspaces to show that they are determined by $\gp$, except for a maximal isotropic subspace and $W$ under the assumption that $\cF$ has a closed isotropic subspace $L$ with $\dim_\C L^\perp / L = 2$.  We must also show that we can determine from $\gp$ whether or not $\cF$ has a closed isotropic subspace $L$ with $\dim_\C L^\perp / L = 2$. 

Let $\widetilde{\gp}$ denote the normalizer in $\gso(\infty;\C)$ of $\gp$.
We use the classical identification $\gso(\infty;\C) \cong \Lambda^2(V)$ where $u\wedge v$
corresponds to the linear transformation $x \mapsto 
\langle x,v\rangle u  - \langle x,u\rangle v$.  With this identification, following \cite{DC2} one has
$$
\widetilde{\gp} = \sum_{\alpha \in A \setminus C} 
	F''_\alpha \wedge (F'_\alpha)^\perp + 
	\sum_{\gamma \in C} F''_\gamma \wedge G''_\gamma + \Lambda^2 
	(W).
$$

Let $\alpha \in A$, and let $x \in F''_\alpha \setminus F'_\alpha$.  
Then one may compute
\begin{eqnarray*}
\widetilde{\gp}  \cdot x & = & \Bigl ( \sum_{\alpha \in A \setminus C} 
	F''_\alpha \wedge (F'_\alpha)^\perp + 
	\sum_{\gamma \in C} F''_\gamma \wedge G''_\gamma + \Lambda^2 
	(W) \Bigr ) \cdot x \\
& = &  \Bigl ( \sum_{\alpha \in A \setminus C} 
	F''_\alpha \otimes (F'_\alpha)^\perp + 
	\sum_{\gamma \in C} F''_\gamma \otimes G''_\gamma \Bigr )  \cdot x \\
& = & \Bigl ( \bigcup_{x \notin (F'_\alpha)^{\perp \perp}} F''_\alpha \Bigr ) \cup 
         \Bigl ( \bigcup_{x \notin (G''_\gamma)^\perp} F''_\gamma \Bigr ).
\end{eqnarray*}
As a result
$$
\widetilde{\gp}  \cdot x =
\begin{cases}
F'_\alpha & \textrm{if } \alpha \in A \setminus C \\
F''_\alpha & \textrm{if } \alpha \in C.
\end{cases}
$$

So far we have shown the following.
If $x \in \widetilde{\gp}  \cdot x$, then $F''_\alpha = \widetilde{\gp}  \cdot x$.  
If $x \notin \widetilde{\gp}  \cdot x$, then 
$F''_\alpha = ( \widetilde{\gp}  \cdot x)^{\perp \perp}$.  
Furthermore, if $x \notin M$, then $\widetilde{\gp}  \cdot x$ is not isotropic, unless 
there exists a closed isotropic subspace $L \in \cF$ with 
$\dim_\C L^\perp / L = 2$, and $x$ is an element of $M_1$ or $M_2$.  We now
consider the union of the subspaces  $\widetilde{\gp}  \cdot x$, where the union is taken
over $x \in V$ for which $\widetilde{\gp}  \cdot x$ is isotropic.  If there does not exist $L$
as described, then these subspaces will be the nested isotropic subspaces computed above, and
indeed their union is $M$.  If $L$ exists, then these subspaces will exhaust 
$L$, and furthermore $M_1$ and $M_2$ will both appear in the union.  Hence the 
union of the isotropic subspaces of the form $\widetilde{\gp}  \cdot x$ for $x \in V$ when 
$L$ exists is $L^\perp$.  As a result, if the union of all the isotropic 
subspaces of the form $\widetilde{\gp}  \cdot x$ for $x \in V$ is itself isotropic, then 
we conclude that no such $L$ exists and we have constructed the subspace $M$.  
If that union is not isotropic, then we conclude that there exists a closed 
isotropic subspace $L \in \cF$ with $\dim_\C L^\perp / L = 2$, and the union 
is $L^\perp$.  In the latter 
case, $L$ is recoverable from $\gp$, as it equals $L^{\perp \perp}$.
We have now shown that we can determine whether $\cF$ has a closed 
isotropic subspace $L$ with $\dim_\C L^\perp / L = 2$, 
that $F''_\alpha$ is determined by $\gp$ for all $\alpha \in A$ in the 
latter case, and that $F_\alpha''$ is determined by $\gp$ for all 
$\alpha \in A$ such that $F''_\alpha \subset L$ in the former case.

We now turn our attention to a non--closed subspace $G''_\gamma$ for $\gamma \in C$.  Since $G''_\gamma$ is not closed, the codimension of $F''_\gamma$ in $G''_\gamma$ is infinite.  Thus if there exists $L \in \cF$ as above, then $F''_\gamma \subset L$.  So we have already shown that $F''_\gamma$,  and indeed $F'_\gamma$ as well, are recoverable from $\gp$ whether or not there exists $L \in \cF$.  Let $x \in (F'_\gamma)^\perp \setminus (F''_\gamma)^\perp$.
Then there exists 
$v \in F''_\gamma$ such that $\langle v ,x \rangle \neq 0$, and one has
$$
(v \wedge G''_\gamma) \cdot x  =  \{ (v \wedge y) \cdot x \mid y \in G''_\gamma \} 
 =  \{ \langle x , y \rangle v - \langle x , v \rangle y \mid y \in G''_\gamma \}.
$$
Since $v \wedge G''_\gamma \subseteq \widetilde{\gp} $ and $v \in F''_\gamma$, we see that 
$G''_\gamma = (v \wedge G''_\gamma) \cdot x + F''_\gamma \subset 
\widetilde{\gp}  \cdot x + F''_\gamma \subset G''_\gamma$.  Hence 
$G''_\gamma =  \widetilde{\gp}  \cdot x + F''_\gamma$, and we conclude that $G''_\gamma$ is recoverable from $\gp$.

Finally, we must show that $\gp$ determines $W$ under the assumption that no subspace $L \in \cF$ as above exists.  We have already shown that $M$ is recoverable from $\gp$ under this assumption.
If $M = M^\bot$, then $W = 0$.   We claim that $W = \widetilde{\gp} \cdot x + M$ for any $x \in M^\perp \setminus M$ when $M \neq M^\bot$.
Indeed, let $X$ be any vector space complement of $M$ in $W$. 
Since $x \notin M$ and $W^\perp = M$, one has $\langle x , X \rangle \neq 0$.  
Furthermore, the restriction of the symmetric bilinear form on $V$ to $X$ is symmetric and nondegenerate. 
Then $\Lambda^2 (X) \cdot x = X$ because $\dim_\C X \geqq 3$.
Since $\Lambda^2(X) \subset \widetilde{\gp}$, we conclude that $\widetilde{\gp}  \cdot x + M = W$.  
Thus $W$ can be recovered from $\gp$.

If $\cF$ is a self--taut generalized flag without any isotropic subspace $L \in \cF$ such that $\dim_\C L^\perp / L = 2$, then we have now shown that $\cF$ is uniquely determined by $\gp$.  Finally, suppose that there does exist an isotropic subspace $L \in \cF$ such that $\dim_\C L^\perp / L = 2$.  Then we have shown that every subspace of $\cF$ which does not lie strictly between $L$ and $L^\perp$ is determined by $\gp$.  There are exactly two maximal isotropic subspaces $M_1$ and $M_2$ containing $L$, and both $M_1$ and $M_2$ are stable under the $\gso(\infty;\C)$--stabilizer of $L$.  Hence the three self-taut generalized flags listed in the statement are precisely the self--taut generalized flags whose stabilizers equal the stabilizer of $\cF$.
\end{proof}

\section{Real Parabolic Subalgebras}
\label{sec3}
\setcounter{equation}{0}

Recall that  a \emph{parabolic subalgebra} of a real Lie algebra $\gg_\R$
is a subalgebra whose complexification is a parabolic subalgebra of the
complexified algebra $\gg_\C$. 

Let $\gg_\C$ be one of
$\ggl(\infty,\C)$, $\gsl(\infty,\C)$, $\gso(\infty,\C)$, and $\gsp(\infty,\C)$,
and let $\gg_\R$ be a real form of $\gg_\C$.  Let $G_\R$ be the corresponding connected real subgroup of $G_\C$.
When $\gg_\R$ has two inequivalent defining representations, we denote them by $V_\R$ and $W_\R$, and when $\gg_\R$ has only one defining representation, we denote it by $V_\R$.  Let $\D$ denote the algebra of $\gg_\R$--endomorphisms of $V_\R$.

\begin{theorem} \label{real-parab}
Suppose that $\gg_\R$ has two inequivalent defining representations.
A subalgebra of $\gg_\R$ (resp. subgroup of $G_\R$) is 
parabolic if and only if it is defined by infinite trace conditions 
(resp. infinite determinant conditions) on the 
$\gg_\R$--stabilizer (resp. $G_\R$--stabilizer) of 
a taut couple of $\D$--generalized flags $\cF$ in $V_\R$ and $\cG$ in $W_\R$.

Suppose that $\gg_\R$ has only one defining representation.
A subalgebra of $\gg_\R$ (resp. subgroup) of $G_\R$ is 
parabolic if and only if it is defined by infinite trace conditions 
(resp. infinite determinant conditions) on the 
$\gg_\R$--stabilizer (resp. $G_\R$--stabilizer) of 
a self--taut $\D$--generalized flag $\cF$ in $V_\R$.
\end{theorem}

\begin{proof}
We will prove the statements for the Lie algebras in question.  The statements on the level of Lie ind--groups follow immediately, since infinite determinant conditions on a Lie ind--group are equivalent to infinite trace conditions on its Lie algebra.  

Suppose that $\gp_\R$ is a parabolic subalgebra of $\gg_\R$.  By definition, the complexification $\gp_\C$ is a parabolic subalgebra of $\gg_\C$.  Theorem~\ref{complex-parab} implies that $\gp_\C$ is defined by infinite trace conditions $\Tr^{\gp_\C}$ on the $\gg_\C$--stabilizer of a taut couple of generalized flags in $V$ and $W$ or on a self--taut generalized flag in $V$.  As $\Tr^{\gp_\C}$
is stable under complex conjugation it is the complexification of the real subspace $(\Tr^{\gp_\C})_\R := \{t \in \Tr^{\gp_\C} \mid \tau(t) = t\}$ where
$\tau$ comes from complex conjugation of $\gg_\C$ over $\gg_\R$.  We will 
use this to show case by case that $\gp_\R$ is defined by trace conditions on the $\gg_\R$--stabilizer of the appropriate generalized flag(s).

The first cases we treat are those where the defining representation space
$V_\R$ is the fixed point set of a complex conjugation $\tau : V \to V$. 
The real forms fitting this description are
$\gsl(\infty;\R)$,
$\gso(\infty,\infty)$,
$\gso(p,\infty)$,
$\gsp(\infty;\R)$, and
$\ggl(\infty;\R)$.
Consider the $\gsl(\infty;\R)$ case, and note that the proof also holds in the $\ggl(\infty;\R)$ case. 
Let $\cF$ and $\cG$ be the taut couple of generalized flags in $V$ and $W$ given in Theorem~\ref{complex-parab}, and note that $W_\R$ is the fixed points of complex conjugation $\tau : W \to W$. Evidently $\tau(\gp_\C) = \gp_\C$, so $\tau(\cF) = \cF$ and $\tau(\cG) = \cG$ by the uniqueness claim of Theorem \ref{complex-parab}. 
Since the generalized flags $\cF$ and $\cG$ are $\tau$--stable, every subspace in them is $\tau$--stable.  (Explicitly, for any $F \in \cF$, we have $\tau (F) \in \cF$, so either $\tau (F) \subset F$ or $F \subset \tau (F)$.  Since $\tau^2 = \mathrm{Id}$, we have $F = \tau (F)$ for any $F \in \cF$.)  Hence every subspace in $\cF$ and $\cG$ has a real form, obtained as the intersection with $V_\R$ and $W_\R$, respectively.
The generalized flags $\cF_\R := \{ F \cap V_\R \mid F \in \cF \}$ and $\cG_\R := \{ G \cap W_\R \mid G \in \cG \}$ form a taut couple as $\R$--generalized flags in $V_\R$ and $W_\R$.  
Now $\gp_\R$ is defined by the infinite trace conditions $(\Tr^{\gp_\C})_\R$ on the $\gsl(\infty;\R)$--stabilizer of the taut couple $\cF_\R$ and $\cG_\R$ of generalized flags in $V_\R$ and $W_\R$.

If $\gg_\R$ is $\gso(*,\infty)$ or $\gsp(\infty;\R)$, Theorem~\ref{complex-parab} implies that $\gp_\C$ is defined by infinite trace conditions on the $\gg_\C$--stabilizer a self--taut generalized flag $\cF$ in $V$.
The arguments of the $\gsl(\infty;\R)$ case show that $\cF$ is $\tau$--stable, 
provided that $\tau (\gp_\C) = \gp_\C$ forces 
$\tau(\cF) = \cF$.  That is ensured by the uniqueness claim in Theorem~\ref{complex-parab} for the
symplectic case, and by Theorem \ref{soc-flag-anomaly} in the orthogonal cases where uniqueness holds.  Uniqueness fails precisely when $\gg_\R =
\gso(\infty,\infty)$ and there exists an isotropic subspace $L \in \cF$ with  $\dim_\C (L^\perp / L) = 2$.
We may assume that $\cF$ is the first of the three generalized flags listed in the statement of Theorem~\ref{soc-flag-anomaly}.  Then $\tau (\cF)$ is one of the three generalized flags listed in the statement of Theorem~\ref{soc-flag-anomaly}, and since $\cF$ is contained in any of those three, the subspaces of $\cF$ are all $\tau$--stable.  Finally, the generalized flag $\cF_\R := \{ F \cap V_\R \mid F \in \cF \}$ in $V_\R$ is self--taut, and $\gp_\R$ is defined by the infinite trace conditions $(\Tr^{\gp_\C})_\R$ on its $\gg_\R$--stabilizer.

Second, suppose that $\gg_\R = \gsu(*,\infty)$.  Note that the arguments for $\gsu(*,\infty)$ apply without change to $\gu(*,\infty)$.  
By Theorem~\ref{complex-parab}, $\gp_\C$ is
given by infinite trace conditions $\Tr^{\gp_\C}$ on the 
$\ggl(\infty;\C)$--stabilizer of a taut couple $\cF$ and $\cG$ of generalized flags in $V$ and $W$.  
There exists an isomorphism of $\gg_\R$--modules $f: V \to W$. 
Both $\cG$ and $f(\cF)$ are stabilized by $\gp_\R$, hence also by $\gp_\C$, so the uniqueness claim of Theorem~\ref{complex-parab} tells us that $\cG = f(\cF)$.
Thus $\cF$ is self--taut.
We conclude that $\gp_\R$ is given by the infinite trace conditions $(\Tr^{\gp_\C})_\R$ on the stabilizer of the self--taut generalized flag $\cF$.

The third case we consider is that of $\gg_\R = \gsl(\infty;\H)$. Note that the $\ggl(\infty;\H)$ case is proved in the same manner.
Then $\gg_\C = \gsl(2\infty;\C)$, where we have the identifications $V = \C^{2\infty} = \C^\infty + \C^\infty j = \H^\infty = V_\R$ and $W = W_\R$.
The quaternionic scalar multiplication $v \mapsto vj$ is a complex
conjugate--linear transformation $J$ of $\C^{2\infty}$ of square $-\mathrm{Id}$,
and the complex conjugation $\tau$ of $\gg_\C$ over $\gg_\R$ is given
by $\xi \mapsto J \xi J^{-1} = J^{-1} \xi J$.
Let $\cF$ and $\cG$ be the unique taut couple given by Theorem~\ref{complex-parab}.  
Since $\gp_\C = \tau(\gp_\C)$, we have $\cF = J(\cF)$ and $\cG = J(\cG)$.  Since $J^2 = -\mathrm{Id}$, every subspace of $\cF$ and $\cG$ is preserved by $J$.  In other words $\cF$ and $\cG$
consist of $\H$-subspaces of $V_\R$ and $W_\R$.  The fact that 
$\cF$ and $\cG$ form a taut couple of $\C$--generalized flags in 
$V$ and $W$ implies via Lemma~\ref{F-independent} that they form a 
taut couple  of $\H$--generalized flags in $V_\R$ and $W_\R$.
Hence $\gp_\R$ is defined by the infinite trace conditions $(\Tr^{\gp_\C})_\R$ on the stabilizer of the taut couple $\cF$, $\cG$. 

The fourth case we consider is that of $\gsp(*,\infty)$.
Then $V_\R$ has an invariant quaternion--hermitian form of signature $(*,\infty)$ and a complex conjugate--linear transformation $J$ of square $-\mathrm{Id}$ as described above.
Let $\cF$ be the unique self--taut generalized flag in $V$ given by Theorem~\ref{complex-parab}.
By the uniqueness of $\cF$, we have $\cF = J(\cF)$,
so as before $\cF$ consists of $\H$--subspaces of $V_\R$.  Lemma~\ref{F-independent} implies that $\cF$ is self--taut when considered as an $\H$--generalized flag in $V_\R$.  Hence $\gp_\R$ is defined by the infinite trace conditions $(\Tr^{\gp_\C})_\R$ on the stabilizer of $\cF$.

The fifth and final case and is that of $\gg_\R = \gso^*(2\infty)$.
Any subspace of $V$ which is stable under  the $\C$-conjugate linear 
map $J$ which corresponds to $x \mapsto xj$ is an $\H$--subspace of $V_\R$.
Let $\cF$ be a self--taut generalized flag in $V$ as given by 
Theorem~\ref{complex-parab}.  Since $\gg_\C = \gso(\infty;\C)$, 
Theorem \ref{soc-flag-anomaly} says that either $\cF$ is unique or there 
are exactly three possibilities for $\cF$.  When $\cF$ is unique, we 
must have $\cF = J(\cF)$, so $\cF$ is an $\H$--generalized flag.
When $\cF$ is not unique, we may assume that $\cF$ is the first of the 
three generalized flags listed in the statement of Theorem 
\ref{soc-flag-anomaly}, the one with an immediate predecessor--successor 
pair $L \subset L^\perp$ where $L$ is closed and $\dim_\C (L^\perp / L) = 2$.
Then $J(\cF)$ has the same property so $J(\cF) = \cF$.  In all cases 
Lemma~\ref{F-independent} implies that $\cF$ is self--taut when considered 
as an $\H$--generalized flag.  Hence $\gp_\R$ is defined by the infinite 
trace conditions $(\Tr^{\gp_\C})_\R$ on the $\gso^*(2\infty)$--stabilizer 
of the self--taut $\H$--generalized flag $\cF$.

Conversely, suppose that $\gp_\R$ is defined by infinite trace conditions $\Tr^{\gp_\R}$ on the $\gg_\R$--stabilizer of a taut couple $\cF_\R$, $\cG_\R$ or a self--taut generalized flag $\cF_\R$, as appropriate.  Either $V = V_\R \otimes \C$ or $V = V_\R$.

Suppose first that $V = V_\R \otimes \C$.  Let $\cF : = \{ F \otimes \C \mid F \in \cF_\R \}$.  If $\gg_\C$ has only one defining representation $V$, then $\cF$ is a self--taut generalized flag in $V$, and $\gp_\C$ is 
defined by the infinite trace conditions $\Tr^{\gp_\R} \otimes \C$ on the $\gg_\C$--stabilizer of $\cF$.  
Now suppose that $\gg_\C$ has two inequivalent defining representations.  If $\gg_\R$ also has two inequivalent defining representations, let $\cG := \{ G \otimes \C \mid G \in \cG_\R \}$.  If $\gg_\R$ has only one defining representation, then let $\cG$ be the image of $\cF$ under the $\gg_\R$--module isomorphism $V \to W$.  Then $\cF$, $\cG$ are a taut couple, and $\gp_\C$ is defined by the infinite trace conditions $\Tr^{\gp_\R} \otimes \C$ on the $\gg_\C$--stabilizer of $\cF$, $\cG$.

Suppose that $V = V_\R$.  Then $\gg_\R$ and $\gg_\C$ have the same number of defining representations.  If $\gg_\R$ has two defining representations, then Lemma~\ref{F-independent} implies that $\cF_\R$ and $\cG_\R$ are a taut couple when considered as $\C$--generalized flags.  Then $\gp_\C$ is defined by the infinite trace conditions $\Tr^{\gp_\R} \otimes \C$ on the $\gg_\C$--stabilizer of $\cF_\R$, $\cG_\R$.
If $\gg_\R$ has only one defining representation, then Lemma~\ref{F-independent} implies that $\cF_\R$ is a self--taut generalized flag when considered as a $\C$--generalized flag.  Thus $\gp_\C$ is 
defined by the infinite trace conditions $\Tr^{\gp_\R} \otimes \C$ on the $\gg_\C$--stabilizer of $\cF_\R$.

In each case, Theorem~\ref{complex-parab} implies that $\gp_\C$ is a parabolic subalgebra of $\gg_\C$, so by definition $\gp_\R$ is a parabolic subalgebra of $\gg_\R$.  
\end{proof}

\begin{theorem}
Let $\gp_\R$ be a parabolic subalgebra of $\gg_\R$.  If $\gg_\R \ncong \gso(\infty,\infty)$, then there is a unique taut couple or self--taut generalized flag associated to $\gp_\R$ by {\rm Theorem~\ref{real-parab}}.  The real analogue of {\rm Theorem~\ref{soc-flag-anomaly}} holds for $\gg_\R \cong \gso(\infty,\infty)$.
\end{theorem}

\begin{proof}
If there is a unique taut couple or self--taut generalized flag associated to $\gp_\C$, then the uniqueness of the taut couple or self--taut generalized flag associated to $\gp_\R$ is immediate from the proof of Theorem~\ref{real-parab}.  If $\gg_\R \cong \gso(\infty,\infty)$, then each of the $\C$--generalized flags of Theorem~\ref{soc-flag-anomaly} has a real form, hence the real analogue of Theorem~\ref{soc-flag-anomaly} holds in this case.  Now suppose that $\gg_\R \cong \gso^*(2\infty)$ and the self--taut generalized flag $\cF$ associated to $\gp_\C$ has a closed isotropic subspace $L$ with $\dim_\C (L^\perp / L) =2$.  The proof of Theorem~\ref{real-parab} shows that $L$ and $L^\perp$ are $\H$--subspaces, and the quaternionic codimension of $L$ in $L^\perp$ is $1$.  Hence the $\H$--generalized flag associated to $\gp_\R$ has no subspaces strictly between $L$ and $L^\perp$, which forces it to be unique.
\end{proof}

\begin{remark}\label{finite-unitary}{\rm
Theorem~\ref{real-parab} simplifies sharply in the $\gsu(p,\infty)$, $\gso(p,\infty)$, $\gsp(p,\infty)$, and $\gu(p,\infty)$ cases when $p \in \Z_{\geqq 0}$.
Because $p$ is the maximal dimension of an isotropic
subspace of $V_\R$ (and thus the maximal codimension of a closed coisotropic subspace), a self--taut generalized flag must be finite.  No infinite trace conditions arise.
The stabilizer of such a self--taut generalized flag coincides with the joint stabilizer of its
isotropic subspaces and at most one non--closed coisotropic subspace.  (The perpendicular complement of the single non--closed coisotropic subspace, when it occurs, is the largest isotropic subspace.)
}
\hfill $\diamondsuit$
\end{remark}

\begin{remark}\label{sec3rem}{\rm
The special case where the subalgebra of $\gg_\C$ (or $\gg_\R$) is 
a direct limit of parabolics of the $\gg_{n,\C}$ (or the $\gg_{n,\R}$)
has been studied in a number of contexts such as \cite{DPW} and \cite{NRW}, 
and in particular in connection with  direct limits of principal 
series representations \cite{W5}.  Any direct limit of
parabolic subalgebras is a parabolic subalgebra in the general sense of this paper.
}
\hfill $\diamondsuit$
\end{remark} 

\section{A Geometric Interpretation} \label{sec5}
\setcounter{equation}{0}

Our geometric interpretation is modeled on a criterion from the  finite--dimensional case.  
Let $G_\C$ be a finite--dimensional classical Lie ind--group, and $G_\R$ a real form of $G_\C$.
Let $P \subset G_\C$ be a parabolic subgroup, and
let $Z := G_\C/P$ be the corresponding flag manifold. 
Then $G_\R$ acts on $Z$ as a subgroup of $G_\C$.  One knows \cite[Theorem 3.6]{W2}
that there is a unique closed $G_\R$--orbit 
$F$ on $Z$, and that 
$\dim_\R F \geqq \dim_\C Z$,
with equality precisely when $F$ is a real form of $Z$.  Thus
real and complex dimensions satisfy $\dim_\R F = \dim_\C Z$
if and only if $F$ is a totally real submanifold of $Z$.
This is the motivation for our geometric interpretation, for 
$F$ is a totally real submanifold of $Z$ if and only if $G_\R$
has a parabolic subgroup whose complexification is 
$G_\C$--conjugate to $P$.  Then that real parabolic subgroup is
the $G_\R$--stabilizer of a point of the closed orbit $F$.
Here note that if any $G_\R$--orbit in $Z$ is totally real then
it has real dimension $\leqq \dim_\C Z$, so it must be the closed orbit.

Let now $G_\C$ be one of the Lie ind--groups $GL(\infty;\C)$, $SL(\infty;\C)$, 
$SO_0(\infty;\C)$ and $Sp(\infty;\C)$.  Fix an exhaustion of $G_\C$ by 
classical connected finite--dimensional subgroups $G_{n,\C}$, and let 
$G_{n,\R}$ be nested real forms of $G_{n,\C}$.  Then 
$G_\R := \varinjlim G_{n,\R}$ is a real form of $G_\C$.
Let $P_\C$ be a parabolic subgroup of $G_\C$.  As 
described in Section \ref{sec2C},
$P_\C$ is defined by infinite determinant conditions on the stabilizer 
$\widetilde{P_\C}$ of a taut couple 
or a self--taut generalized flag.  Here $\widetilde{P_\C}$ is the normalizer
of $P_C$ in $G_C$.
We use the usual notation for the Lie algebras of all these Lie ind--groups.

\begin{lemma} \label{totreal1}
Consider the homogeneous space $Z = G_\C/\widetilde{P_\C}$.  
Write $z_0$ for the identity coset $1\cdot\widetilde{P_\C}$ in $Z$ and 
define $Z_n = G_{n,\C}(z_0)$.  Then each $Z_n$ 
is a (finite--dimensional) complex homogeneous space and 
$Z$ is the complex ind--manifold $\varinjlim Z_n$ (direct limit 
in the category of complex manifolds and holomorphic maps.)
\end{lemma}

\begin{proof} $\widetilde{P_\C}$ is a complex subgroup of $G_\C$, and
$\widetilde{P_\C} = \varinjlim (G_{n,\C}\cap \widetilde{P_\C})$.  Each
finite--dimensional orbit $Z_n$ is a complex manifold because 
$G_{n,\C}\cap \widetilde{P_\C}$ is a complex subgroup of $G_{n,\C}$, and
the inclusions $Z_n \hookrightarrow Z_{n+1}$ are holomorphic embeddings.  
As in \cite{NRW} now $Z = \varinjlim Z_n$ is a strict direct limit in the
category of complex manifolds and holomorphic maps.  In other words a
function $f$ on an open subset $U \subset Z$ is holomorphic if and only if
each of the $f|_{U\cap Z_n} : U\cap Z_n \to \C$ is holomorphic.  Note that
separately holomorphic functions on open subsets $U \subset Z$ are jointly 
holomorphic because each $f|_{U\cap Z_n}$ is jointly holomorphic (and thus
continuous) by Hartogs' Theorem.
\end{proof}

\begin{lemma} \label{totreal2}
Let $Y = G_\R(z_0)$ and $Y_n = G_{n,\R}(z_0)$.  Then $Y$ is
a totally real submanifold of $Z$ if and only if each $Y_n$ is a totally
real submanifold of $Z_n$.
\end{lemma}

\begin{proof}
Let $J$ denote the complex structure operator
for $Z$, linear transformation of square $-\mathrm{Id}$ on the complexified tangent
space $T:=T_{z_0,\C}(Z)$ of $Z$ at $z_0$.  Then $J$ preserves each of the
$T_n:=T_{z_0,\C}(Z_n)$.  Now $Y$ is totally real if and only if the real
tangent space $T_\R := T_{z_0}(Z)$ satisfies $J(T_\R) \cap T_\R =0$, and
$Y_n$ is totally real if and only if the real tangent space $T_{n,\R}
:= T_{z_0}(Z_n)$ satisfies $J(T_{n,\R}) \cap T_{n,\R} = 0$.  Since
$T_\R = \varinjlim T_{n,\R}$ the assertion follows.
\end{proof}

\begin{lemma} \label{totreal3}
$G_{n,\R}\cap \widetilde{P_\C}$ is a real form of 
$G_{n,\C}\cap \widetilde{P_\C}$ if and only if $Y_n$ is totally real in $Z_n$.
\end{lemma}

\begin{proof}
Denote $H_{n,\C} = G_{n,\C}\cap \widetilde{P_\C}$ and 
$H_{n,\R} = G_{n,\R}\cap \widetilde{P_\C}$.
Suppose first that $Y_n$ is totally real in $Z_n$.  Then 
$\dim_\R G_{n,\R} - \dim_\R H_{n,\R} = \dim_\R Y_n \leqq \dim_\C Z_n
= \dim_\C G_{n,\C} - \dim_\C H_{n,\C}$, so $\dim_\R H_{n,\R} \geqq
\dim_\C H_{n,\C}$, forcing $\dim_\R H_{n,\R} = \dim_\C H_{n,\C}$.  
Now $H_{n,\R}$ is a real form of $H_{n,\C}$.

Conversely suppose that $H_{n,\R}$ is a real form of $H_{n,\C}$.  Then
the real tangent space to $Y_n$ at $z_0$ is represented by any vector
space complement $\gm_{n,\R}$ to $\gh_{n,\R}$ in $\gg_{n,\R}$, while the
real tangent space to $Z_n$ at $z_0$ is represented by the vector
space complement $\gm_{n,\R}\otimes\C$ to $\gh_{n,\C}$ in $\gg_{n,\C}$,
so $Y_n$ is totally real in $Z_n$. 
\end{proof}

Putting all this together, we have our geometric characterization of
parabolic subgroups of the classical real Lie ind--groups.

\begin{theorem}\label{characterization}
Fix a parabolic subgroup $P_\C \subset G_\C$ and consider the 
flag ind--manifold $Z = G_\C/\widetilde{P_\C}$.  Then $P_\C\cap G_\R$
is a parabolic subgroup of $G_\R$ if and only if the following two conditions hold:

{\rm (i)} the orbit $G_\R(z_0)$ of the base point $z_0 = \widetilde{P_\C}$ 
is a totally real submanifold of $Z$;

{\rm (ii)} the set of all infinite trace
conditions on $\widetilde{\gp_\C}$ satisfied by $\gp_\C$ is stable under the
complex conjugation $\tau$ of $\gg_\C$ over $\gg_\R$.
\end{theorem}

\begin{proof}
Lemmas \ref{totreal2} and \ref{totreal3} show that the orbit
$G_\R(z_0)$ is a totally real submanifold of $Z$
if and only if $G_\R \cap \widetilde{P_\C}$ is parabolic in $G_\R$.

If $G_\R \cap P_\C$ is parabolic in $G_\R$ then $G_\R \cap \widetilde{P_\C}$ 
is parabolic
because it contains $G_\R \cap P_\C$, and the corresponding real set of 
infinite trace 
conditions complexifies to the set of infinite trace conditions by which
$\gp_\C$ is defined from $\widetilde{\gp_\C}$.  Thus (i) and (ii) follow.

Conversely assume (i) and (ii).  From (i), $G_\R \cap \widetilde{P_\C}$
is a parabolic subgroup of $G_\R$, and from (ii),
$\{x \in \gg_\R \cap \widetilde{\gp_\C} \mid x \text{ satisfies }
\Tr^{\gp_\C}\}\otimes \C = \{x \in \widetilde{\gp_\C} \mid x \text{ satisfies }
\Tr^{\gp_\C}\}$, where $\Tr^{\gp_\C}$ denotes the set of infinite trace 
conditions described in Definition \ref{def-trace-cond}.
\end{proof}

{\small

\vskip .1 in

\centerline{\begin{tabular}{lll}
E.D.-C.: & I.P.: & J.A.W.: \\
Department of Mathematics & School of Engineering and Science & Department of Mathematics \\
Rice University & Jacobs University Bremen  & University of California  \\
6100 S. Main St. & Campus Ring 1 & Berkeley, CA 94720--3840 \\
Houston TX 77005-1892 & 28759 Bremen, Germany &  \\
{\tt edc@rice.edu} & {\tt i.penkov@jacobs-university.de} & {\tt jawolf@math.berkeley.edu}
\end{tabular}}
}
\end{document}